\documentclass[12pt,twoside]{amsart}

\date{}

\begin{document}

\centerline {\Large{\bf Generalised intuitionistic fuzzy soft sets and its }}
\centerline {\Large{\bf application in decision making}}
\centerline{}

\newcommand{\mvec}[1]{\mbox{\bfseries\itshape #1}}

\centerline{\bf {Bivas Dinda, Tuhin Bera and T.K. Samanta}}

\centerline{}

\newtheorem{Theorem}{\quad Theorem}[section]

\newtheorem{definition}[Theorem]{\quad Definition}

\newtheorem{theorem}[Theorem]{\quad Theorem}

\newtheorem{remark}[Theorem]{\quad Remark}

\newtheorem{corollary}[Theorem]{\quad Corollary}

\newtheorem{note}[Theorem]{\quad Note}

\newtheorem{lemma}[Theorem]{\quad Lemma}

\newtheorem{example}[Theorem]{\quad Example}

\newtheorem{proposition}[Theorem]{\quad Proposition}

\centerline{}
\centerline{\bf Abstract}
{\emph{In this paper, generalised intuitionistic fuzzy soft sets and relations on generalised intuitionistic fuzzy soft sets are defined and a few of their properties are studied. An application of generalised intuitionistic fuzzy soft sets in decision making with respect to degree of preference is investigated.}}\\

\section{\bf Introduction}
In real life situation, most of the problems in economics, social science, medical science, environment etc. have various uncertainties. However, most of the existing mathematical tools for formal modeling, reasoning and computing are crisp, deterministic and precise in character. There are theories viz. theory of probability, evidence, fuzzy set, intuitionistic fuzzy set, vague set, interval mathematics, rough set for dealing with uncertainties. These theories have their own difficulties as pointed out by Molodtsov\cite{Molodtsov}.\\
In 1999, Molodtsov\cite{Molodtsov} initiated a novel concept of soft set theory, which is completely new approach for modeling vagueness and uncertainties. Soft set theory has a rich potential for application in solving practical problems in economics, social science, medical science etc. Later on Maji et al. \cite{Maji2,Maji,Maji1} have studied the theory of fuzzy soft set and intuitionistic fuzzy soft set. Majumder and Samanta \cite{sksamanta} have generalised the concept of fuzzy soft set as introduced by Maji et al. \cite{Maji2}.\\
As a generalisation of fuzzy soft set theory, intuitionistic fuzzy set theory makes description of the objective world more realistic, practical and accurate in some cases, making it very promising. The motivation of the present paper is to further generalise the concept of Majumder and Samanta \cite{sksamanta}. \\
In this paper, we have introduced generalised intuitionistic fuzzy soft set. Our definition is more realistic since it contains a degree of preference corresponding to each parameter. Relations on generalised intuitionistic fuzzy soft sets  are defined and a few of its properties are studied. An application of generalised intuitionistic fuzzy soft set in decision making is presented.

\section{\bf Preliminaries}
\begin{definition}\cite{Molodtsov}
Let $U$ be an initial universe set and $E$ be the set of parameters. Let $P(U)$ denotes the power set of $U$. A pair $(F,E)$ is called a soft set over $U$ where $F$ is a mapping given by $F:E\rightarrow P(U)$.
\end{definition}

\begin{definition}\cite{Maji2}
Let $U$ be an initial universe set and $E$ be the set of parameters. Let $A \subset E$. A pair $(F,A)$ is called fuzzy soft set over $U$ where $F$
 is a mapping given by $F : A\rightarrow I^U  $, where $I^U$ denotes the collection of all fuzzy subsets of $U$.
 \end{definition}

\begin{definition}\cite{sksamanta}
Let $U=\{x_1,x_2,\cdots,x_n\}$ be the universal set of elements and $E=\{e_1,e_2,\cdots,e_m\}$ be the universal set of parameters. The pair $(U,E)$ will be called a soft universe. Let $F:E\rightarrow I^U$ and $\mu$ be a fuzzy subset of $E,i.e. \;\mu:E\rightarrow I=[0,1],$ where $I^U$ is the collection of all fuzzy subset of $U.$ Let $F_{\mu}$ be a mapping $F_{\mu}:E\rightarrow I^U\times I$ defined as follows:\\
$F_{\mu}(e)=(F(e),\mu(e))$, where $F(e)\in I^{U}.$ Then $F_{\mu}$ is called generalised fuzzy soft set over the soft universe $(U,E).$
 \end{definition}
Here for each parameter $e_i,\;F_{\mu}(e_i)$ indicates not only the degree of belongingness of the elements of $U$ in $F(e_i)$ but also the degree of possibility of such belongingness which is represented by $\mu(e_i)$.

\begin{definition}\cite{Maji}
Let $U$ be an initial universe set and $E$ be the set of parameters. Let $IF^U$ denotes the collection of all intuitionistic fuzzy subsets of $U$. Let $A\subseteq E$. A pair $(F,\,A)$ is called intuitionistic fuzzy soft set over $U$, where $F$ is a mapping given by $F:\,A\rightarrow \,IF^U.$
\end{definition}

\begin{example}
Consider the following example:\\
Let $(F,A)$ describes the the character of the students with respect to the given parameters, for finding the best student of an academic year. Let the set of students under consideration is $U=\{s_1,s_2,s_3,s_4\}.$ Let $A\subseteq E$ and $A=\{r="result",\,c="conduct",\,g="games\; and\; sports\; performances"\}$. Let\\
$F(r)=\left\{(s_1,0.8,0.1),(s_2,0.9,0.05),(s_3,0.85,0.1),(s_4,0.75,0.2)\right\}\\$
$F(c)=\left\{(s_1,0.6,0.3),(s_2,0.65,0.2),(s_3,0.7,0.2),(s_4,0.65,0.2)\right\}\\$
$F(g)=\left\{(s_1,0.75,0.2),(s_2,0.5,0.3),(s_3,0.5,0.4),(s_4,0.7,0.2)\right\}\\$
Then the family $\{F(r),F(c),F(g)\}$ of $IF^U$ is an intuitionistic fuzzy soft set.
\end{example}

\begin{definition}\cite{Maji}
Intrersection of two intuitionistic fuzzy soft sets $(F,A)$ and $(G,B)$ over a common universe $U$ is the intuitionistic fuzzy soft set $(H,C)$ where $C=A\cap B,$ and $\forall \epsilon\in C,\;H(e)=F(e)\cap G(e).$ We write $(F,A)\tilde{\cap}(G,B)=(H,C).$
\end{definition}

\begin{definition}\cite{Maji}
Union of two intuitionistic fuzzy soft sets $(F,A)$ and $(G,B)$ over a common universe $U$ is the intuitionistic fuzzy soft set $(H,C)$ where $C=A\cup B,$ and $\forall \epsilon\in C,$
\[H(e)=F(e),\hspace{3 cm}if\; e\in A-B\]
\[=G(e),\hspace{3 cm}if\;e\in B-A\hspace{-1 cm}\]
\[=F(e)\cup G(e),\hspace{1.75 cm}if\; e\in A\cap B\hspace{-1 cm}\]
We write $(F,A)\tilde{\cup}(G,B)=(H,C).$
\end{definition}

\begin{definition}\cite{Maji}
For two intuitionistic fuzzy soft sets $(F,A)$ and $(G,B)$ over a common universe $U$, we say that $(F,A)$ is an intuitionistic fuzzy soft subset of $(G,B)$ if\\
(i) $A\subset B,$ and\\
(ii) $\forall \epsilon\in A,\;F(\epsilon)$ is an intuitionistic fuzzy subset of $G(\epsilon).\\$
We write $(F,A)\tilde{\subset}(G,B).$
\end{definition}

\begin{definition}
\cite{Schweizer} A binary operation \, $\ast \; : \; [\,0 \; , \;
1\,] \; \times \; [\,0 \; , \; 1\,] \;\, \rightarrow \;\, [\,0
\; , \; 1\,]$ \, is continuous \, $t$ - norm if \,$\ast$\, satisfies
the
following conditions \, $:$ \\
(i)  $\ast$ \, is commutative and associative ,\\
(ii) $\ast$ \, is continuous , \\
(iii) $a \;\ast\;1 \;\,=\;\, a \hspace{1.2cm}
\forall \;\; a \;\; \in \;\; [\,0 \;,\; 1\,]$ , \\
(iv) $a \;\ast\; b \;\, \leq \;\, c \;\ast\; d$
\, whenever \, $a \;\leq\; c$  ,  $b \;\leq\; d$  and  $a \,
, \, b \, , \, c \, , \, d \;\, \in \;\;[\,0 \;,\; 1\,]$.
\end{definition}
A few examples of continuous t-norm are $\,a\,\ast\,b\,=\,ab,\;\,a\,\ast\,b\,=\,min\{a,b\},\;\,a\,\ast\,b\,=\,max\{a+b-1,0\}$.

\begin{definition}
\cite{Schweizer}. A binary operation \, $\diamond \; : \; [\,0 \; ,
\; 1\,] \; \times \; [\,0 \; , \; 1\,] \;\, \rightarrow \;\,
[\,0 \; , \; 1\,]$ \, is continuous \, $t$-conorm if
\,$\diamond$\, satisfies the
following conditions \, $:$ \\
(i)  $\diamond$ \, is commutative and
associative ,\\
(ii)  $\diamond$ \, is continuous , \\
(iii) $a \;\diamond\;0 \;\,=\;\, a
\hspace{1.2cm}
\forall \;\; a \;\; \in\;\; [\,0 \;,\; 1\,]$ , \\
(iv) $a \;\diamond\; b \;\, \leq \;\, c
\;\diamond\; d$ \, whenever \, $a \;\leq\; c$  ,  $b \;\leq\; d$
 and  $a \, , \, b \, , \, c \, , \, d \;\; \in\;\;[\,0
\;,\; 1\,].$
\end{definition}
A few examples of continuous t-conorm are $\,a\,\diamond\,b\,=\,a+b-ab,\;\,a\,\diamond\,b\,=\,max\{a,b\},\;\,a\,\diamond\,b\,=\,min\{a+b,1\}$.\\

\section{\bf Generalised intuitionistic fuzzy soft sets}
Throughout the text, unless otherwise stated explicitly, $U$ be the set of universe and $E$ be the set of parameters and we take $A,B,C\subseteq E$ and $\alpha,\beta,\delta$ are fuzzy subset of $A,B,C$ respectively.

\begin{definition}
Let $U$ be the universal set and $E$ be the set of parameters. Let $A\subseteq E$ and $\mathcal{F}:A\rightarrow IF^{U}$ and $\alpha$ be a fuzzy subset of $A$ i.e., $\alpha:A\rightarrow [0,1]$, where $IF^{U}$ is the collection of all intuitionistic fuzzy subset of $U$. Let $\mathcal{F}_\alpha:A\rightarrow IF^{U}\times [0,1]$ be a function defined as follows:\\
\[\mathcal{F}_\alpha (a)\,=\,\left(\mathcal{F}(a)=\{x,\,\mu_{_{\mathcal{F}(a)}}(x),\,\nu_{_{\mathcal{F}(a)}}(x)\},\;\alpha(a)\right)\]
where $\mu,\nu$ denotes the degree of membership and degree of non-membership. \\ Then $\mathcal{F}_\alpha $ is called a Generalised intuitionistic fuzzy soft set over $(U,E).$
\end{definition}
Here for each parameter $e_i,\;\mathcal{F}_\alpha (e_i) $ indicates not only degree of belongingness of the elements of $U$ in $\mathcal{F}(a)$ but also degree of preference of such belongingness which is represented by $\alpha(e_i).$

\begin{example}\label{e1}
Let $U=\{s_1,s_2,s_3,s_4\}$ be the set of students under consideration for the best student of an academic year with respect to the given parameters $A\subseteq E$ and $A=\{r="result",\,c="conduct",\,g="games \;and\; sports\; performances"\}$. Let $\alpha:A\rightarrow [0,1]$ be given as follows: $\alpha(r)=0.7,\;\alpha(c)=0.5,\;\alpha(g)=0.6\\$
We define $\mathcal{F}_\alpha$ as follows:\\
$\mathcal{F}_\alpha(r)=\left(\{(s_1,0.8,0.1),(s_2,0.9,0.05),(s_3,0.85,0.1),(s_4,0.75,0.2)\},\;0.7\right)\\$
$\mathcal{F}_\alpha(c)=\left(\{(s_1,0.6,0.3),(s_2,0.65,0.2),(s_3,0.7,0.2),(s_4,0.65,0.2)\},\;0.5\right)\\$
$\mathcal{F}_\alpha(g)=\left(\{(s_1,0.75,0.2),(s_2,0.5,0.3),(s_3,0.5,0.4),(s_4,0.7,0.2)\},\;0.6\right)\\$
Then $\mathcal{F}_\alpha$ is an generalised intuitionistic fuzzy soft set.
\end{example}

\begin{definition}
Let $\mathcal{F}_\alpha$ and $\mathcal{G}_\beta$ be two generalised intuitionistic fuzzy soft set over $(U,E).$ Now $\mathcal{F}_\alpha$ is called a generalised intuitionistic fuzzy soft subset of $\mathcal{G}_\beta$ if\\
(i) $\alpha$ is a fuzzy subset of $\beta,$\\
(ii) $A\subseteq B,\\$
(iii) $\forall a\in A,\; \mathcal{F}(a)$ is an intuitionistic fuzzy subset of $\mathcal{G}(a)$ i.e., $\mu_{_{\mathcal{F}(a)}}(x)\leq \mu_{_{\mathcal{G}(a)}}(x)$ and $\nu_{_{\mathcal{F}(a)}}(x)\geq \nu_{_{\mathcal{G}(a)}}(x)\;\;\;\forall x\in U$ and $a\in A.\\$
We write $\mathcal{F}_\alpha\tilde{\subseteq}\mathcal{G}_\beta.$
\end{definition}

\begin{example}\label{e2}
Let $\mathcal{G}_\beta$ be a generalised intuitionistic fuzzy soft set defined as follows:\\
$\mathcal{G}_\beta(r)=\left(\{(s_1,0.85,0.05),(s_2,0.9,0.025),(s_3,0.9,0.1),(s_4,0.8,0.1)\},\;0.75\right)\\$
$\mathcal{G}_\beta(c)=\left(\{(s_1,0.7,0.2),(s_2,0.7,0.15),(s_3,0.75,0.2),(s_4,0.65,0.15)\},\;0.6\right)\\$
$\mathcal{G}_\beta(g)=\left(\{(s_1,0.8,0.2),(s_2,0.6,0.3),(s_3,0.7,0.2),(s_4,0.7,0.1)\},\;0.65\right)\\$
and consider the generalised intuitionistic fuzzy soft set $\mathcal{F}_\alpha$ given in Example \ref{e1}. Then $\mathcal{F}_\alpha$ is a generalised intuitionistic fuzzy soft subset of $\mathcal{G}_\beta$.
\end{example}

\begin{definition}
The intersection of two generalised intuitionistic fuzzy soft sets $\mathcal{F}_\alpha$ and $\mathcal{G}_\beta$ is denoted by $\mathcal{F}_\alpha \tilde{\cap}\,\mathcal{G}_\beta$ and defined by a generalised intuitionistic fuzzy soft set $\mathcal{H}_\delta:A\cap B\rightarrow IF^{U}\times [0,1]$ such that for each $e\in A\cap B$ and $x\in U\\$
\[\mathcal{H}_\delta (e)=\left(\{x,\,\mu_{_{\mathcal{H}(e)}}(x),\,\nu_{_{\mathcal{H}(e)}}(x)\},\;\delta(e)\right)\]
where $\;\;\mu_{_{\mathcal{H}(e)}}(x)=\mu_{_{\mathcal{F}(e)}}(x)\ast \mu_{_{\mathcal{G}(e)}}(x),\;\;\; \nu_{_{\mathcal{H}(e)}}(x)=\nu_{_{\mathcal{F}(e)}}(x)\diamond \nu_{_{\mathcal{G}(e)}}(x),\;\;\;\delta(e)=\alpha(e)\ast\beta(e).$
\end{definition}

\begin{definition}
The union of two generalised intuitionistic fuzzy soft sets $\mathcal{F}_\alpha$ and $\mathcal{G}_\beta$ is denoted by $\mathcal{F}_\alpha \tilde{\cup}\,\mathcal{G}_\beta$ and defined by a generalised intuitionistic fuzzy soft set $\mathcal{H}_\delta:A\cup B\rightarrow IF^{U}\times [0,1]$ such that for each $e\in A\cup B$ and $x\in U$
\[\mathcal{H}_\delta (e)=\left(\{x,\,\mu_{_{\mathcal{F}(e)}}(x),\,\nu_{_{\mathcal{F}(e)}}(x)\},\;\alpha(e)\right)\;\;\;if \;e\in A-B\]
\[=\left(\{x,\,\mu_{_{\mathcal{G}(e)}}(x),\,\nu_{_{\mathcal{G}(e)}}(x)\},\;\beta(e)\right)\;\;\;if \;e\in B-A\hspace{-1 cm}\]
\[=\left(\{x,\,\mu_{_{\mathcal{H}(e)}}(x),\,\nu_{_{\mathcal{H}(e)}}(x)\},\;\delta(e)\right)\;\;\;if \;e\in A\cap B\hspace{-1 cm}\]
where $\;\;\mu_{_{\mathcal{H}(e)}}(x)=\mu_{_{\mathcal{F}(e)}}(x)\diamond \mu_{_{\mathcal{G}(e)}}(x),\;\;\; \nu_{_{\mathcal{H}(e)}}(x)=\nu_{_{\mathcal{F}(e)}}(x)\ast \nu_{_{\mathcal{G}(e)}}(x),\;\;\;\delta(e)=\alpha(e)\diamond\beta(e).$
\end{definition}

\begin{example}
Let us consider the generalised intuitionistic fuzzy soft sets $\mathcal{F}_\alpha$ and $\mathcal{G}_\beta$ defined in Example \ref{e1} and \ref{e2} respectively. Let us define the $t$-norm $\ast$ the $t$-conorm $\diamond$ as follows: $a\ast b=ab$ and $a\diamond b=a+b-ab$. Then \\
$(\mathcal{F}_\alpha \tilde{\cup}\,\mathcal{G}_\beta)(e_1)=(\{(s_1,0.97,0.005),(s_2,0.99,0.00125),(s_3,0.985,0.01),\\(s_4,0.95,0.02)\}\,,\;0.68\,)\\$
$(\mathcal{F}_\alpha \tilde{\cup}\,\mathcal{G}_\beta)(e_2)=(\{(s_1,0.88,0.06),(s_2,0.895,0.03), (s_3,0.925,0.04),\\(s_4,0.8775,0.1625)\}\,,\;0.1625\,)\\$
$(\mathcal{F}_\alpha \tilde{\cup}\,\mathcal{G}_\beta)(e_3)=(\{(s_1,0.95,0.04),(s_2,0.8,0.09),(s_3,0.85,0.08),\\(s_4,0.91,0.02)\}\,,\;0.86\,).\\$
Since $\{r,c,g\}\in A\cap B,\\$
$(\mathcal{F}_\alpha \tilde{\cap}\,\mathcal{G}_\beta)(e_1)=(\{(s_1,0.68,0.145),(s_2,0.81,0.07375),(s_3,0.765,0.19),\\(s_4,0.6,0.28)\}\,,\;0.12\,)\\$
$(\mathcal{F}_\alpha \tilde{\cap}\,\mathcal{G}_\beta)(e_2)=(\{(s_1,0.42,0.44),(s_2,0.455,0.32), (s_3,0.525,0.36),\\(s_4,0.4225,0.7375)\}\,,\;0.375\,)\\$
$(\mathcal{F}_\alpha \tilde{\cap}\,\mathcal{G}_\beta)(e_3)=(\{(s_1,0.6,0.36),(s_2,0.3,0.5),(s_3,0.35,0.52),\\(s_4,0.49,0.28)\}\,,\;0.39\,).$
\end{example}

\begin{theorem}
Let $\mathcal{F}_\alpha$, $\mathcal{G}_\beta$ and $\mathcal{H}_\delta$ be any three generalised intuitionistic fuzzy soft sets over $(U,E),$ then the following holds:\\
(i) $\mathcal{F}_\alpha\tilde{\cup}\,\mathcal{G}_\beta=\mathcal{G}_\beta\tilde{\cup}\,\mathcal{F}_\alpha.\\$
(ii) $\mathcal{F}_\alpha\tilde{\cap}\,\mathcal{G}_\beta=\mathcal{G}_\beta\tilde{\cap}\,\mathcal{F}_\alpha.\\$
(iii) $\mathcal{F}_\alpha\tilde{\cup}\,(\mathcal{G}_\beta \tilde{\cup}\,\mathcal{H}_\delta)=(\mathcal{F}_\alpha\tilde{\cup}\,\mathcal{G}_\beta)\tilde{\cup}\,\mathcal{H}_\delta .\\$
(iv) $\mathcal{F}_\alpha\tilde{\cap}\, (\mathcal{G}_\beta\tilde{\cap}\,\mathcal{H}_\delta)=(\mathcal{F}_\alpha\tilde{\cap}\,\mathcal{G}_\beta)\tilde{\cap}\,\mathcal{H}_\delta .$
\end{theorem}
\begin{proof}
 Since the $t$-norm function and $t$-conorm functions are commutative and associative, therefore the theorem follows.
\end{proof}

\begin{remark}
Let $\mathcal{F}_\alpha$, $\mathcal{G}_\beta$ and $\mathcal{H}_\delta$ be any three generalised intuitionistic fuzzy soft sets over $(U,E)$. If we consider $a\ast b=\min\{a,\,b\}$ and $a\diamond b=\max\{a,\,b\}$ then the following holds:\\
(i) $\mathcal{F}_\alpha\tilde{\cap}\,(\mathcal{G}_\beta \tilde{\cup}\,\mathcal{H}_\delta)=(\mathcal{F}_\alpha\tilde{\cap}\,\mathcal{G}_\beta)\tilde{\cup}\,(\mathcal{F}_\alpha\tilde{\cap}\,\mathcal{H}_\delta) \\$
(ii) $\mathcal{F}_\alpha\tilde{\cup}\,(\mathcal{G}_\beta \tilde{\cap}\,\mathcal{H}_\delta)=(\mathcal{F}_\alpha\tilde{\cup}\,\mathcal{G}_\beta)\tilde{\cap}\,(\mathcal{F}_\alpha\tilde{\cup}\,\mathcal{H}_\delta) .\\$
But in general above relations does not hold.
\end{remark}

\section{\bf Relation on generalised intuitionistic fuzzy soft sets}
\begin{definition}
Let $\mathcal{F_\alpha}$ and $\mathcal{G_\beta}$ be two generalised intuitionistic fuzzy soft set over $(U,E).$ Then generalised intuitionistic fuzzy soft relation (in short GIFSR) $R$ from $\mathcal{F_\alpha}$ to $\mathcal{G_\beta}$ is a function $R:A\times B\rightarrow IF^U\times [0,1]$ defined by
\[R(a,b)\tilde{\subseteq}\mathcal{F_\alpha}(a)\tilde{\cap} \mathcal{G_\beta}(b)\;\;\;\;\forall (a,b)\in A\times B. \]
\end{definition}

\begin{definition}
Let $R_1,\,R_2$ be two GIFSR from $\mathcal{F_\alpha}$ to $\mathcal{G_\beta}$. Then $R_1\cup R_2,\;R_1\cap R_2,\;R_1^{-1}$ are defined as follows:\\
$(R_1\cup R_2)(a,b)=\max\{R_1(a,b),\,R_2(a,b)\}.\\$
$(R_1\cap R_2)(a,b)=\min\{R_1(a,b),\,R_2(a,b)\}.\\$
$R_1^{-1}(a,b)=R_1(b,a).\;\;\;\forall (a,b)\in A\times B.$
\end{definition}

\begin{note}
If $R$ is a GIFSR from $\mathcal{F_\alpha}$ to $\mathcal{G_\beta}$ then $R^{-1}$ is a GIFSR from $\mathcal{G_\beta}$ to $\mathcal{F_\alpha}$.
\end{note}

\begin{proposition}
If $R_1$ and $R_2$ are GIFSR from $\mathcal{F_\alpha}$ to $\mathcal{G_\beta}$,\\
(i) ${(R_1^{-1})}^{-1}=R_1.\\$
(ii) $R_1\subseteq R_2\;\Rightarrow\,R_1^{-1}\subseteq\,R_2^{-1}.$
\end{proposition}
\begin{proof}
Let $(a,b)\in A\times B.$ \\
(i) ${(R_1^{-1})}^{-1}(a,b)=R_1^{-1}(b,a)=R_1(a,b).$ Hence ${(R_1^{-1})}^{-1}=R_1.\\$
(ii) $R_1(a,b)\subseteq R_2(a,b)\;\Rightarrow\,{(R_1^{-1})}^{-1}(a,b)\subseteq\,{(R_2^{-1})}^{-1}(a,b)\Rightarrow\,R_1^{-1}(b,a)\subseteq R_2^{-1}(b,a).$ Hence $R_1^{-1}\subseteq R_2^{-1}.$
\end{proof}

\begin{definition}
The composition $\circ$ of two GIFSR $R_1$ and $R_2$ is defined by \[(R_1\circ R_2)(a,c)=R_1(a,b)\tilde{\cap} R_2(b,c)\] where $R_1$ is a relation from $\mathcal{F_\alpha}$ to $\mathcal{G_\beta}$ and $R_2$ is a GIFSR from $\mathcal{G_\beta}$ to $\mathcal{H_\delta}.$
\end{definition}

\begin{theorem}
Let $R_1$ be a GIFSR from $\mathcal{F_\alpha}$ to $\mathcal{G_\beta}$ and $R_2$ be a relation $\mathcal{G_\beta}$ to $\mathcal{H_\delta}$. Then $R_1\circ R_2$ is a GIFSR from $\mathcal{F_\alpha}$ to $\mathcal{H_\delta}$.
\end{theorem}
\begin{proof}
By definition\\
$R_1(a,b)\subseteq \mathcal{F_\alpha}(a)\tilde{\cap}\mathcal{G_\beta}(b)=
\{\left(\{x,\,\mu_{_{\mathcal{F}(a)}}(x)\ast \mu_{_{\mathcal{G}(b)}}(x), \,\nu_{_{\mathcal{F}(a)}}(x)\diamond \nu_{_{\mathcal{G}(b)}}(x)\},\;\alpha(a)\ast\beta(b)\right)\,:\;x\in U \},\;\forall (a,b)\in A\times B .\\$
$R_2(b,c)\subseteq \mathcal{G_\beta}(b)\tilde{\cap}\mathcal{H_\delta}(c)=\{\left(\{x,\,\mu_{_{\mathcal{G}(b)}}(x)\ast \mu_{_{\mathcal{H}(c)}}(x), \,\nu_{_{\mathcal{G}(b)}}(x)\diamond \nu_{_{\mathcal{H}(c)}}(x)\},\;\beta(b)\ast\delta(c)\right)\,:\;x\in U \},\;\forall (b,c)\in B\times C.$
Therefore,\\
$(R_1\circ R_2)(a,c)=R_1(a,b)\tilde{\cap} R_2(b,c)\\=\{(\{x,\;(\mu_{_{\mathcal{F}(a)}}(x)\ast \mu_{_{\mathcal{G}(b)}}(x))\;\ast\;(\mu_{_{\mathcal{G}(b)}}(x)\ast \mu_{_{\mathcal{H}(c)}}(x)),
\;(\nu_{_{\mathcal{F}(a)}}(x)\diamond \nu_{_{\mathcal{G}(b)}}(x))\;\diamond\;(\nu_{_{\mathcal{G}(b)}}(x)\diamond \nu_{_{\mathcal{H}(c)}}(x))\},
\;(\alpha(a)\ast\beta(b))\,\ast\,(\beta(b)\ast\delta(c)))\,:\;x\in U \},\;\forall (a,b,c)\in A\times B\times C\,.\;$
Now\\ $(\mu_{_{\mathcal{F}(a)}}(x)\ast \mu_{_{\mathcal{G}(b)}}(x))\;\ast\;(\mu_{_{\mathcal{G}(b)}}(x)\ast \mu_{_{\mathcal{H}(c)}}(x))\\= \mu_{_{\mathcal{F}(a)}}(x)\ast \mu_{_{\mathcal{G}(b)}}(x)\ast \mu_{_{\mathcal{H}(c)}}(x)\\ \leq \mu_{_{\mathcal{F}(a)}}(x)\ast\, 1\,\ast \mu_{_{\mathcal{H}(c)}}(x)\\= \mu_{_{\mathcal{F}(a)}}(x)\ast \mu_{_{\mathcal{H}(c)}}(x)\,\;$ and \\
$(\nu_{_{\mathcal{F}(a)}}(x)\diamond \nu_{_{\mathcal{G}(b)}}(x))\;\diamond\;(\nu_{_{\mathcal{G}(b)}}(x)\diamond \nu_{_{\mathcal{H}(c)}}(x))\\= \nu_{_{\mathcal{F}(a)}}(x)\diamond \nu_{_{\mathcal{G}(b)}}(x)\diamond \nu_{_{\mathcal{H}(c)}}(x)\\ \geq \nu_{_{\mathcal{F}(a)}}(x)\diamond\, 0\,\diamond \nu_{_{\mathcal{H}(c)}}(x)\\= \nu_{_{\mathcal{F}(a)}}(x)\diamond \nu_{_{\mathcal{H}(c)}}(x).\;$ \\Also,
$(\alpha(a)\ast\beta(b))\,\ast\,(\beta(b)\ast\delta(c))=\alpha(a)\ast\beta(b)\ast\delta(c)\leq \alpha(a)\ast \,1\,\ast\delta(c)=\alpha(a)\ast\delta(c).\\$
Hence $R_1(a,b)\tilde{\cap} R_2(b,c)\subseteq \mathcal{F_\alpha}\tilde{\cap}\mathcal{H_\delta}.\\$
Thus $\;R_1\circ R_2$ is a GIFSR from $\mathcal{F_\alpha}$ to $\mathcal{H_\delta}$.
\end{proof}

\begin{proposition}
$R_1\circ (R_2\cup R_3)=(R_1\circ R_2)\cup(R_1\circ R_3)$ where $R_1$ is a GIFSR from $\mathcal{F_\alpha}$ to $\mathcal{G_\beta}$ and $R_2,\,R_3$ are GIFSR from $\mathcal{G_\beta}$ to $\mathcal{H_\delta}.$
\end{proposition}

\begin{proof}
 Let $a\in A,\;b\in B,\;c\in C.\\$
$R_1(a,b)\circ \left(R_2(b,c)\cup R_3(b,c)\right)\\=R_1(a,b)\tilde{\cap}\; \max\{R_2(b,c),\; R_3(b,c)\} \\= \max\{R_1(a,b)\tilde{\cap} \,R_2(b,c),\; R_1(a,b)\tilde{\cap}\,R_3(b,c)\}\\=\max\{(R_1\circ R_2)(a,c),\; (R_1 \circ R_3)(a,c)\}\\=(R_1\circ R_2)(a,c)\cup (R_1\circ R_3)(a,c).\\$ So, $R_1\circ (R_2\cup R_3)=(R_1\circ R_2)\cup(R_1\circ R_3)$.
\end{proof}

\begin{proposition}
$(R_1\circ R_2)^{-1}=R_2^{-1}\circ R_1^{-1}$ where $R_1$ is a IFSR from $\mathcal{F_\alpha}$ to $\mathcal{G_\beta}$ and $R_2$ are GIFSR from $\mathcal{G_\beta}$ to $\mathcal{H_\delta}.$
\end{proposition}

\begin{proof}
 Let $a\in A,\;b\in B,\;c\in C.\\$
$(R_1 \circ R_2)^{-1}(c,a)=(R_1 \circ R_2)(a,c)=R_1(a,b) \tilde{\cap}\,R_2(b,c)=R_2(b,c)\tilde{\cap}\,R_1(a,b)=R_2^{-1}(c,b)\tilde{\cap}\,
 R_1^{-1}(b,a)=(R_2^{-1}\circ R_1^{-1})(c,a).\;$ Hence $(R_1\circ R_2)^{-1}=R_2^{-1}\circ R_1^{-1}\,.$
\end{proof}
\[\]
\section{\bf An application of generalised intuitionistic fuzzy soft set in decision making }
There are several applications of generalised intuitionistic fuzzy soft set theory in several directions. Here we present application of generalised intuitionistic fuzzy soft set in a decision making problem. Suppose there are six boys in the universe $U$ as $U=\{b_1,b_2,b_3,b_4,b_5,b_6\}$ and the parameter set $E=\{e_1,e_2,e_3,e_4,e_5,e_6,e_7,e_8,e_9\}$, where each $e_i,\;1\leq i\leq 9$ indicates a specific criteria for boys.\\
$e_1$ stands for "education qualification".\\
$e_2$ stands for "hard working".\\
$e_3$ stands for "responsible".\\
$e_4$ stands for "government employee".\\
$e_5$ stands for "non-government employee".\\
$e_6$ stands for "businessman".\\
$e_7$ stands for "family status".\\
$e_8$ stands for "spiritual and ideal".\\
$e_9$ stands for "handsome".\\
Suppose a woman Miss. Y wishes to marry a man on the basis of her wishing parameter among the listed above. Our aim is to find out the most appropriate partner for Miss. Y.\\
Suppose the wishing parameters of Miss. Y be $A\subseteq E$ where $A=\{e_3,e_4,e_7,e_{9}\}\,.\\$
Let $\alpha:A\rightarrow [0,1]$ be a fuzzy subset of $A$, defined by Miss. Y as follows:\\
$\alpha(e_3)=0.1,\;\;\alpha(e_4)=0.5,\;\;\alpha(e_7)=0.4,\;\;\alpha(e_{9})=0.3\;.\\$
Consider the generalised intuitionistic fuzzy soft sets $\mathcal{F}_\alpha$ as a collection of intuitionistic fuzzy approximation as below:

$\mathcal{F}_\alpha (e_3)=(\{(b_1,0.3,0.5),(b_2,0.5,0.3),(b_3,0.3,0.4),(b_4,0.6,0.3),(b_5,0.4,0.3),\\(b_6,0.2,0.4)\}, \;0.1\,)\\$
$\mathcal{F}_\alpha (e_4)=(\{(b_1,0,0.8),(b_2,1,0),(b_3,0.9,0.02),(b_4,0,0.12),(b_5,0,0.2),\\(b_6,0,0.03)\},\; 0.5\,)\\$
$\mathcal{F}_\alpha (e_7)=(\{(b_1,0.6,0.3),(b_2,0.5,0.4), (b_3,0.6,0.35),(b_4,0.7,0.2),(b_5,0.7,0.28),\\(b_6,0.8,0.02)\},\; 0.4\,)\\$
$\mathcal{F}_\alpha (e_{9})=(\{(b_1,0.5,0.3),(b_2,0.4,0.3), (b_3,0.6,0.38),(b_4,0.5,0.3),(b_5,0.5,0.2),\\(b_6,0.7,0.19)\},\; 0.3\,)\\$

Now we introduce the following operations:\\
(i) for membership function: $\mu^{\prime}_{b_r}(e_i)=a_i+b_i-a_i b_i,\;\;\;\;$ where $\;a_i=\mu_{b_r}(e_i)$ and $b_i=\alpha(e_i)$ for $r=1,2,3,4,5,6$.\\
(ii) for non-membership function: $\nu\,^{\prime}_{b_r}(e_i)=c_i d_i,\;\;\;\;$ where $\;c_i=\nu_{b_r}(e_i)$ and $d_i=\alpha(e_i)$ for $r=1,2,3,4,5,6$.\\
Actually we have taken these two operations to ascend the membership value and descend the non-membership value of $\mathcal{F}_\alpha (e_i)$ on the basis of the degree of preference of Miss. Y. Then the generalised intuitionistic fuzzy soft set $\mathcal{F}_\alpha (e_i)$ reduced to an intuitionistic fuzzy soft set $\mathcal{F}^{\prime} (e_i)$ given as follows:\\
$\mathcal{F}^{\prime}(e_3)=\{(b_1,0.37,0.05),(b_2,0.55,0.03),(b_3,0.37,0.04),(b_4,0.64,0.03),\\(b_5,0.46,0.03),(b_6,0.28,0.04)\}\\$
$\mathcal{F}^{\prime}(e_4)=\{(b_1,0.5,0.4),(b_2,1,0),(b_3,0.95,0.01),(b_4,0.5,0.06),\\(b_5,0.5,0.1),(b_6,0.5,0.15)\}\\$
$\mathcal{F}^{\prime} (e_7)=\{(b_1,0.76,0.12),(b_2,0.7,0.16),(b_3,0.76,0.1),(b_4,0.82,0.08),\\(b_5,0.82,0.112),(b_6,0.88,0.008)\}\\$
$\mathcal{F}^{\prime}(e_{9})=\{(b_1,0.65,0.09),(b_2,0.58,0.09),(b_3,0.78,0.114),(b_4,0.65,0.09),\\(b_5,0.65,0.06),(b_6,0.79,0.057)\}.$

\begin{definition}({\bf Comparison table})
It is a square table in which number of rows and number of column are equal and both are labeled by the object name of the universe such as $b_1,b_2,\cdots,b_n$ and the entries are $c_{ij},\;\;$ where \\$c_{ij}=$ the number of parameters for which the value of $b_i$ exceeds or equal to the value of $b_j$.
\end{definition}

{\bf Algorithm:\\}
(i) Input the set $A\subseteq E$ of choice of parameters of Miss. Y.\\
(ii) Consider the reduced intuitionistic fuzzy soft set in tabular form.\\
(iii) Compute the comparison table of membership function and non-membership function.\\
(iv) Compute the membership score and non-membership score.\\
(v) Compute the final score by subtracting non-membership score from membership score.\\
(vi) Find the maximum score, if it occurs in i-th row then Miss. Y will marry to $b_i$.

\begin{table}
[tbh]
\begin{center}
\begin{tabular}{|r||*{3}{l@{}l|}r|r|} \hline
\multicolumn{1}{|c|}{$\cdot$} & \multicolumn{2}{c|}{$e_3$} &
\multicolumn{2}{c|}{$e_4$} &\multicolumn{2}{c|}{$e_7$} &
\multicolumn{1}{c|}{$e_{9}$} \\
\hline\hline
$b_1$ & 0.37 &  & 0.5 &  & 0.76 &  & 0.65 \\
$b_2$ & 0.55 &  & 1.0 &  & 0.7 &  & 0.58  \\
$b_3$ & 0.37 &  & 0.95 &  & 0.76 &  & 0.78  \\
$b_4$ & 0.64 &  & 0.5 &  & 0.82 &  & 0.65  \\
$b_5$ & 0.46 &  & 0.5 &  & 0.82 &  & 0.65  \\
$b_6$ & 0.28 &  & 0.5 &  & 0.88 &  & 0.79  \\\hline
\end{tabular}
\end{center}
\caption{Tabular representation of membership function}
\end{table}

\begin{table}
[tbh]
\begin{center}
\begin{tabular}{|r||*{5}{l@{}l|}r|r|} \hline
\multicolumn{1}{|c|}{$\cdot$} & \multicolumn{2}{c|}{$b_1$} &
\multicolumn{2}{c|}{$b_2$}    &\multicolumn{2}{c|}{$b_3$} &
\multicolumn{2}{c|}{$b_4$}    &\multicolumn{2}{c|}{$b_5$} &
\multicolumn{1}{c|}{$b_6$} \\
\hline\hline
$b_1$ & 4 && 2 && 2 && 2 && 2 && 2  \\
$b_2$ & 2 && 4 && 2 && 1 && 2 && 2  \\
$b_3$ & 4 && 2 && 4 && 2 && 2 && 2  \\
$b_4$ & 4 && 3 && 2 && 4 && 4 && 2  \\
$b_5$ & 4 && 2 && 2 && 3 && 4 && 2  \\
$b_6$ & 3 && 2 && 2 && 3 && 3 && 4 \\ \hline
\end{tabular}
\end{center}
\caption{Comparison table of the above table}
\end{table}

\begin{table}
[tbh]
\begin{center}
\begin{tabular}{|r||*{3}{l@{}l|}r|r|r|} \hline
\multicolumn{1}{|c|}{$\cdot$} & \multicolumn{2}{c|}{Row sum(a)} &
\multicolumn{2}{c|}{Column sum(b)} &\multicolumn{2}{c|}{Membership score(a-b)}\\
\hline\hline
$b_1$ & 14 && 21 &&  -7&\\
$b_2$ & 13 && 15 && -2 &\\
$b_3$ & 16 && 14 && 2  &\\
$b_4$ & 19 && 15 && 4  &\\
$b_5$ & 17 && 17 && 0  &\\
$b_6$ & 17 && 14 && 3  &\\\hline
\end{tabular}
\end{center}
\caption{Membership score table}
\end{table}

\begin{table}
[tbh]
\begin{center}
\begin{tabular}{|r||*{3}{l@{}l|}r|r|} \hline
\multicolumn{1}{|c|}{$\cdot$} & \multicolumn{2}{c|}{$e_3$} &
\multicolumn{2}{c|}{$e_4$} &\multicolumn{2}{c|}{$e_7$} &
\multicolumn{1}{c|}{$e_{9}$} \\
\hline\hline
$b_1$ & 0.05 &  & 0.40 &  & 0.12 &  & 0.09 \\
$b_2$ & 0.03 &  & 0.00 &  & 0.16 &  & 0.09  \\
$b_3$ & 0.04 &  & 0.01 &  & 0.10 &  & 0.114  \\
$b_4$ & 0.03 &  & 0.06 &  & 0.08 &  & 0.09 \\
$b_5$ & 0.03 &  & 0.10 &  & 0.112 &  & 0.06  \\
$b_6$ & 0.04 &  & 0.015 &  & 0.008 &  & 0.057  \\\hline
\end{tabular}
\end{center}
\caption{Tabular representation of non-membership function}
\end{table}

\begin{table}
[tbh]
\begin{center}
\begin{tabular}{|r||*{5}{l@{}l|}r|r|} \hline
\multicolumn{1}{|c|}{$\cdot$} & \multicolumn{2}{c|}{$b_1$} &
\multicolumn{2}{c|}{$b_2$}    &\multicolumn{2}{c|}{$b_3$} &
\multicolumn{2}{c|}{$b_4$}    &\multicolumn{2}{c|}{$b_5$} &
\multicolumn{1}{c|}{$b_6$} \\
\hline\hline
$b_1$ & 4 && 3 && 3 && 4 && 4 && 4  \\
$b_2$ & 2 && 4 && 1 && 3 && 2 && 2  \\
$b_3$ & 1 && 3 && 4 && 3 && 2 && 3  \\
$b_4$ & 1 && 3 && 1 && 4 && 2 && 3  \\
$b_5$ & 0 && 2 && 2 && 3 && 4 && 3  \\
$b_6$ & 0 && 2 && 2 && 1 && 1 && 4  \\\hline
\end{tabular}
\end{center}
\caption{Comparison table of the above table}
\end{table}

\begin{table}
[tbh]
\begin{center}
\begin{tabular}{|r||*{3}{l@{}l|}r|r|r|} \hline
\multicolumn{1}{|c|}{$\cdot$} & \multicolumn{2}{c|}{Row sum(c)} &
\multicolumn{2}{c|}{Column sum(d)} &\multicolumn{2}{c|}{Non-membership score(c-d)}\\
\hline\hline
$b_1$ & 22 && 8 &&  14&\\
$b_2$ & 14 && 17 && -3 &\\
$b_3$ & 16 && 13 && 3  &\\
$b_4$ & 14 && 18 && -4  &\\
$b_5$ & 14 && 15 && -1 &\\
$b_6$ & 10 && 19 && -9  &\\\hline
\end{tabular}
\end{center}
\caption{Non-membership score table}
\end{table}

\begin{table}
[tbh]
\begin{center}
\begin{tabular}{|r||*{3}{l@{}l|}r|r|r|} \hline
\multicolumn{1}{|c|}{$\cdot$} & \multicolumn{2}{c|}{Membership score(m)} &
\multicolumn{2}{c|}{Non-membership score(n)} &\multicolumn{2}{c|}{Finale score(m-n)}\\
\hline\hline
$b_1$ & -7 && 14 &&  -21&\\
$b_2$ & -2 && -3 && 1 &\\
$b_3$ & 2 && 3 && -1  &\\
$b_4$ & 4 && -4 && 8  &\\
$b_5$ & 0 && -1 && 1  &\\
$b_6$ & 3 && -9 && 12  &\\\hline
\end{tabular}
\end{center}
\caption{Final score table}
\end{table}
\[\]\[\]\[\]\[\]\[\]\[\]\[\]\[\]\[\]\[\]
Clearly the maximum score is $12$ scored by the man $b_6$.\\
{\bf Decision:} Miss. Y will marry to $b_6$. In case, if she does not want to marry $b_6$ due to certain reasons, her second choice will be $b_4$.
\section{\bf Conclusion}
In this paper, we have introduced the weighted intuitionistic fuzzy soft sets and soft relations with respect to preference. An application of this theory to solve a socialistic problem in a different approach has been investigated. It is expected that the approach will be useful to handle several realistic uncertain problems and give more perfect results.

{\[\]\bf Bivas Dinda}\\
Department of Mathematics, \\
Mahishamuri Ramkrishna Vidyapith,\\
P.O.-Nowpara, Amta, Howrah,\\
West Bengal, India.\\
e-mail: bvsdinda@gmail.com\\\\
 {\bf Tuhin Bera}\\
Department of Mathematics,\\
Boror Siksha Satra High School,\\
West Bengal, India. \\
e-mail: tuhinor@gmail.com\\\\
 {\bf T.K.Samanta}\\
Department of Mathematics,\\
 Uluberia College,\\
 West Bengal, India.\\
e-mail: mumpu$_{-}$tapas5@yahoo.co.in
\end{document}